\documentclass[twoside,10pt]{article}%
\usepackage{amssymb}
\usepackage{amsfonts}
\usepackage{amsmath}
\usepackage{graphicx}%
\providecommand{\U}[1]{\protect\rule{.1in}{.1in}}
 \arraycolsep=1.5pt 

\newtheorem{theorem}{Theorem}

\newtheorem{corollary}[theorem]{Corollary}

\newtheorem{lemma}[theorem]{Lemma}

\newtheorem{remark}[theorem]{Remark}

\newenvironment{proof}[1][Proof]{\noindent\textbf{#1.} }{\ \rule{0.5em}{0.5em}}
\begin{document}

\title{\textbf{Blowup Phenomena for Compressible Euler Equations with Non-vacuum
Initial Data}}
\author{\textsc{Sen Wong\thanks{E-mail address: senwongsenwong@yahoo.com.hk} and
}M\textsc{anwai Yuen\thanks{Corresponding Author, E-mail address:
nevetsyuen@hotmail.com }}\\\textit{Department of Mathematics and Information Technology,}\\\textit{The Hong Kong Institute of Education,}\\\textit{10 Lo Ping Road, Tai Po, New Territories, Hong Kong}}
\date{Revised 20-Apr-2015}
\maketitle

\begin{abstract}
In this article, we study the blowup phenomena of compressible Euler equations
with non-vacuum initial data. Our new results, which cover a general class of
testing functions, present new initial value blowup conditions. The
corresponding blowup results of the $1$-dimensional case in non-radial
symmetry are also included.

\ 


\ 


\end{abstract}

\section{Introduction and Main Results}

$N$-dimensional compressible isentropic Euler equations for fluids can be
expressed as%
\begin{equation}
\left\{
\begin{array}
[c]{rl}%
{\normalsize \rho}_{t}{\normalsize +\nabla\cdot(\rho u)} & {\normalsize =}%
{\normalsize 0}\\
\rho\lbrack u_{t}+(u\cdot\nabla)u]{\normalsize +\nabla}P & {\normalsize =}%
0\text{,}%
\end{array}
\right.  \label{Euler}%
\end{equation}
where $\rho=\rho(t,x):[0,\infty)\times R^{N}\rightarrow\lbrack0,\infty)$,
$u=u(t,x):[0,\infty)\times R^{N}\rightarrow R^{N}$ and $P$ are the density,
the velocity, and the pressure functions respectively. For polytropic fluids,
we have
\begin{equation}
P=K\rho^{\gamma}, \label{gamma}%
\end{equation}
for which the constants $K>0$ and $\gamma\geq1$.

For non-vacuum initial data, the density remains positive for $t\geq0$. From
equation $(1)_{1}$, we know that the value of $\rho(t,x)$ is determined by
$\rho_{0}(x)$ and an exponential function along a characteristic curve. More
precisely, we have the following lemma.

\begin{lemma}
If $\rho_{0}(x)>0$ for all $x\in R^{N}$, then $\rho(t,x)>0$ for all $t\geq0$
and for all $x\in R^{N}$.
\end{lemma}

\begin{proof}
With the material derivative along a characteristic curve $x(t;x_{0})$, the
mass equation $(1)_{1}$ becomes%
\begin{equation}
\frac{\mathrm{D}\rho}{\mathrm{D}t}+\rho\bigtriangledown\cdot u=0\text{.}%
\end{equation}
Taking the integration, we obtain%
\begin{equation}
\rho(t,x(t;x_{0}))=\rho_{0}(x_{0})exp\left(  -\int_{0}^{t}\bigtriangledown
\cdot u(s,x(s;x_{0}))ds\right)  \text{.}%
\end{equation}
The result follows easily from the above equation.
\end{proof}

In radial symmetry, Equations ($\ref{Euler}$) are written in the following
form%
\begin{equation}
\left\{
\begin{array}
[c]{c}%
\rho_{t}+V\rho_{r}+\rho V_{r}+\dfrac{N-1}{r}\rho V=0\\
\rho\left(  V_{t}+VV_{r}\right)  +P_{r}=0\text{.}%
\end{array}
\right.  \label{eqRS}%
\end{equation}
Here,
\begin{equation}
\rho=\rho(t,r)\text{ \ \ \ and \ \ \ }u=\frac{x}{r}V(t,r)=:\frac{x}%
{r}V\text{,}%
\end{equation}
with the radius $r=\left(  \sum_{i=1}^{N}x_{i}^{2}\right)  ^{1/2}$.

For the development of and classical results of the Euler equations and fluid
mechanics, readers may refer to $\cite{LE,CH,Lions1,E,MUK,SI,SI2,TS,Zhu,LZ}$.

In contrast to the condition given in $\cite{TS}$, where a vacuum state is
considered, we investigate the Euler equations with a non-vacuum state and the
finite propagation is applied. By refining the arguments in $\cite{SI,STW}$,
we obtain the corresponding result for $R^{N}$ using the following lemma.

\begin{lemma}
\label{tt2}Let $(\rho,u)$ be a $C^{1}$ solution of the $N$-dimensional Euler
equations $(\ref{Euler})$ with $\gamma>1$, life span $T>0$ and the following
initial data:%
\begin{equation}
\left\{
\begin{matrix}
(\rho(0,x),u(0,x))=\left(  \bar{\rho}+\rho_{0}(x),u_{0}(x)\right) \\
\text{supp}(\rho_{0},u_{0})\subseteq\{x:|x|\leq R\}\text{,}%
\end{matrix}
\right.
\end{equation}
for some positive constants $\bar{\rho}$ and $R$. Then, we have%
\begin{equation}
(\rho,u)=(\bar{\rho},0)
\end{equation}
for $t\in\lbrack0,T)$ and $|x|\geq R+\sigma t$, where $\sigma=\sqrt
{K\gamma\bar{\rho}^{\gamma-1}}>0$.
\end{lemma}

\begin{proof}
The proof is included in the Appendix.
\end{proof}

The following corollary is the radial symmetry version of Lemma $\ref{tt2}$.

\begin{corollary}
\label{c4}Let $(\rho,V)$ be a $C^{1}$ solution of the $N$-dimensional Euler
equations in radial symmetry ($\ref{eqRS}$) with $\gamma>1$, life span $T>0$
and the following initial data%
\begin{equation}
\left\{
\begin{matrix}
(\rho(0,r),V(0,r))=(\bar{\rho}+\rho_{0}(r),V_{0}(r))\label{initial}\\
\text{supp}(\rho_{0},V_{0})\subseteq\{r:r\leq R\}\text{,}%
\end{matrix}
\right.
\end{equation}
for some positive constants $\bar{\rho}$ and $R$. Then, we have%
\begin{equation}
(\rho,V)=(\bar{\rho},0)\text{,}%
\end{equation}
for $t\in[0,T)$ and $r\geq R+\sigma t$, where $\sigma=\sqrt{K\gamma\bar{\rho
}^{\gamma-1}}>0$.
\end{corollary}

In 2011, Yuen obtained the initial functional conditions for the blowup of the
Euler-Poisson equations for testing functions $f(r)=r^{n}$ (with $n=1$ in
$\cite{Yuen1}$ and an arbitrary positive constant $n$ in $\cite{Yuen2}$).
Subsequently, the authors in $\cite{Gf}$ designed general testing functions to
obtain the initial functional conditions for showing the blowup phenomena of
the Euler and Euler-Poisson equations using the integration method under the
nonslip boundary condition $\cite{Ns}$. Recently, the authors in $\cite{LMZ}$
obtained improved blowup results for the Euler and Euler-Poisson equations
with repulsive forces based on $\cite{Yuen1}$. To apply the integration
method, controlling of the support of the data is required. With the
assistance of Corollary $\ref{c4}$, we can use the integration method to study
the blowup phenomena of the Euler equations in which the nonslip boundary
condition is replaced by an initial value condition, and thus obtain new
blowup results. More precisely, we have the following theorems.

\begin{theorem}
\label{tt6}Fix $a>2$ and $\tau>0$. Let $f(r)$ be a strictly increasing $C^{1}$
function that vanishes at $0$. Under the setting of Corollary $\ref{c4}$, if
$H_{1}(0)$ is large enough such that%
\begin{equation}
\frac{(a-2)H_{1}^{2}(0)}{2aB_{1}(\tau)}-\frac{K\gamma}{\gamma-1}\bar{\rho
}^{\gamma-1}f(R+\sigma\tau)>0 \label{9}%
\end{equation}
and%
\begin{equation}
H_{1}(0)\geq\left[  \int_{0}^{\tau}\frac{1}{aB_{1}(s)}ds\right]  ^{-1}\text{,}
\label{10}%
\end{equation}
where%
\begin{equation}
H_{1}(t)=\int_{0}^{\infty}f(r)V(t,r)dr
\end{equation}
and%
\begin{equation}
B_{1}(t)=\int_{0}^{R+\sigma t}\frac{f^{2}(r)}{{f}^{\prime}(r)}dr\text{,}%
\end{equation}
then, the time $T<\tau$.
\end{theorem}

\begin{theorem}
\label{tt7}Fix $a>2$ and $\tau>0$. Let $f(x)$ be a non-negative strictly
increasing $C^{1}$ function. Under the setting of Lemma $\ref{tt2}$ with
$N=1$, if $H_{2}(0)$ is large enough such that%
\begin{equation}
\frac{(a-2)H_{2}^{2}(0)}{2aB_{2}(\tau)}-\frac{K\gamma}{\gamma-1}\bar{\rho
}^{\gamma-1}f(R+\sigma\tau)>0
\end{equation}
and%
\begin{equation}
H_{2}(0)\geq\left[  \int_{0}^{\tau}\frac{1}{aB_{2}(s)}ds\right]  ^{-1}\text{,}
\label{14}%
\end{equation}
then the time $T<\tau$, where%
\begin{equation}
H_{2}(t)=\int_{-\infty}^{+\infty}f(x)u(t,x)dx
\end{equation}
and%
\begin{equation}
B_{2}(t)=\int_{-R-\sigma t}^{R+\sigma t}\frac{f^{2}(x)}{{f}^{\prime}%
(x)}dx\text{.}%
\end{equation}

\end{theorem}

Other blowup results for the compressible Euler equations are provided in
Section 2.

\section{Integration Methods}

First, we give a detailed proof of Theorem $\ref{tt6}$ using the integration
method for $\gamma>1$ as follows.

\begin{proof}
[Proof of Theorem $\ref{tt6}$]Equation ($\ref{eqRS}$)$_{2}$, for non-vacuum
initial data, becomes%
\begin{equation}
V_{t}+\partial_{r}(\frac{1}{2}V^{2})+\frac{K\gamma}{\gamma-1}\partial_{r}%
(\rho^{\gamma-1}-\bar{\rho}^{\gamma-1})=0\text{.} \label{eq19}%
\end{equation}
Multiplying equation ($\ref{eq19}$) by function $f(r)$ and taking the
integration over $[0,\infty)$, we get%
\begin{equation}
\dot{H_{1}}(t)+\int_{0}^{\infty}f(r)\partial_{r}(\frac{1}{2}V^{2}%
)dr+\frac{K\gamma}{\gamma-1}\int_{0}^{\infty}f(r)\partial_{r}(\rho^{\gamma
-1}-\bar{\rho}^{\gamma-1})dr=0\text{.} \label{keyeq}%
\end{equation}
Note that the integrals are well defined.\newline Using the integration by
parts, we get%
\begin{equation}
\dot{H_{1}}(t)+\int_{0}^{R+\sigma t}f(r)\partial_{r}(\frac{1}{2}V^{2}%
)dr+\frac{K\gamma}{\gamma-1}\int_{0}^{R+\sigma t}f(r)\partial_{r}(\rho
^{\gamma-1}-\bar{\rho}^{\gamma-1})dr=0
\end{equation}%
\begin{align}
&  \dot{H_{1}}(t)+\frac{1}{2}\left[  V^{2}(t,r)f(r)\right]  _{0}^{R+\sigma
t}+\frac{K\gamma}{\gamma-1}\left[  (\rho^{\gamma-1}-\bar{\rho}^{\gamma
-1})f(r)\right]  _{0}^{R+\sigma t}\nonumber\\
&  =\frac{1}{2}\int_{0}^{R+\sigma t}V^{2}{f}^{\prime}(r)dr+\frac{K\gamma
}{\gamma-1}\int_{0}^{R+\sigma t}(\rho^{\gamma-1}-\bar{\rho}^{\gamma-1}%
){f}^{\prime}(r)dr
\end{align}%
\begin{equation}
\dot{H_{1}}(t)=\frac{1}{2}\int_{0}^{R+\sigma t}V^{2}{f}^{\prime}%
(r)dr+\frac{K\gamma}{\gamma-1}\int_{0}^{R+\sigma t}(\rho^{\gamma-1}-\bar{\rho
}^{\gamma-1}){f}^{\prime}(r)dr \label{28}%
\end{equation}%
\begin{equation}
\geq\frac{1}{2}\int_{0}^{R+\sigma t}V^{2}{f}^{\prime}(r)dr-\frac{K\gamma
}{\gamma-1}\int_{0}^{R+\sigma t}\bar{\rho}^{\gamma-1}{f}^{\prime}(r)dr
\end{equation}%
\begin{equation}
=\frac{1}{2}\int_{0}^{R+\sigma t}V^{2}{f}^{\prime}(r)dr-\frac{K\gamma}%
{\gamma-1}\bar{\rho}^{\gamma-1}{f}(R+\sigma t)\text{.}%
\end{equation}
That is,%
\begin{equation}
\dot{H_{1}}(t)\geq\frac{1}{2}\int_{0}^{R+\sigma t}V^{2}{f}^{\prime}%
(r)dr-\frac{K\gamma}{\gamma-1}\bar{\rho}^{\gamma-1}f(R+\sigma t)\text{.}%
\end{equation}
\newline On the other hand, by the Cauchy Inequality,%
\begin{equation}
\left[  \int_{0}^{R+\sigma t}Vf(r)dr\right]  ^{2}\leq\int_{0}^{R+\sigma
t}V^{2}{f}^{\prime}(r)dr\int_{0}^{R+\sigma t}\frac{f^{2}(r)}{{f}^{\prime}%
(r)}dr
\end{equation}%
\begin{equation}
\int_{0}^{R+\sigma t}V^{2}{f}^{\prime}(r)dr\geq\frac{H_{1}^{2}(t)}{B_{1}%
(t)}\text{.}%
\end{equation}
Hence,%
\begin{equation}
\dot{H_{1}}(t)\geq\frac{H_{1}^{2}(t)}{2B_{1}(t)}-\frac{K\gamma}{\gamma-1}%
\bar{\rho}^{\gamma-1}f(R+\sigma t)\text{.}%
\end{equation}
When $0\leq t\leq\tau$, we have%
\begin{equation}
\dot{H_{1}}(t)\geq\frac{H_{1}^{2}(t)}{aB_{1}(t)}+\left[  \frac{(a-2)H_{1}%
^{2}(t)}{2aB_{1}(t)}-\frac{K\gamma}{\gamma-1}\bar{\rho}^{\gamma-1}f(R+\sigma
t)\right]
\end{equation}%
\begin{equation}
\geq\frac{H_{1}^{2}(t)}{aB_{1}(t)}+\left[  \frac{(a-2)H_{1}^{2}(t)}%
{2aB_{1}(\tau)}-\frac{K\gamma}{\gamma-1}\bar{\rho}^{\gamma-1}f(R+\sigma
\tau)\right]
\end{equation}%
\begin{equation}
=:\frac{H_{1}^{2}(t)}{aB_{1}(t)}+G_{1}(t)\text{.}%
\end{equation}
From condition ($\ref{9}$), we have $G_{1}(0)>0$. It follows that
$G_{1}(t)\geq0$ for $0\leq t\leq\tau$. More precisely, suppose $G_{1}%
(t_{1})<0$, for some $0<t_{1}\leq\tau$, then there exists a constant $t_{2}$,
where $0<t_{2}<t_{1}$, such that%
\begin{equation}
\left\{
\begin{matrix}
G_{1}(t)>0\text{,} & 0\leq t<t_{2}\\
G_{1}(t)=0\text{,} & t=t_{2}\\
G_{1}(t)<0\text{,} & t_{2}<t<t_{2}+\varepsilon_{1}\text{,}%
\end{matrix}
\right.
\end{equation}
for some $\varepsilon_{1}>0$.\newline Thus, $\dot{H_{1}}(t_{2})\geq0$ implies
$H_{1}(t_{2}+\varepsilon_{2})\geq H_{1}(t_{2})>0$, for some $0<\varepsilon
_{2}<\varepsilon_{1}$. Thus, $G_{1}(t_{2}+\varepsilon_{2})\geq G_{1}(t_{2}%
)=0$, which is a contradiction.\newline Therefore, for $0\leq t\leq\tau$, we
have%
\begin{equation}
H_{1}(t)\geq H_{1}(0)>0
\end{equation}
and%
\begin{equation}
\dot{H_{1}}(t)\geq\frac{H_{1}^{2}(t)}{aB_{1}(t)}\text{.}%
\end{equation}
It follows that for $0\leq t\leq\tau$, we obtain%
\begin{equation}
\frac{1}{H_{1}(0)}-\frac{1}{H_{1}(t)}\geq\int_{0}^{t}\frac{1}{aB_{1}%
(s)}ds\text{.}%
\end{equation}
Thus,%
\begin{equation}
0<\frac{1}{H_{1}(t)}\leq\frac{1}{H_{1}(0)}-\int_{0}^{t}\frac{1}{aB_{1}%
(s)}ds\text{.}%
\end{equation}
From condition ($\ref{10}$), we conclude that the non-vacuum solutions for the
Euler equations ($\ref{eqRS}$) blow up before $\tau$, that is, the time
$T<\tau$.\newline The proof is complete.
\end{proof}

Second, the proof of Theorem $\ref{tt7}$ for the corresponding $1$-dimensional
case in non-radial symmetry is presented.

\begin{proof}
[Proof of Theorem $\ref{tt7}$]The $1$-dimensional momentum equation
($\ref{Euler}$)$_{2}$ with non-vacuum data is written as%
\begin{equation}
u_{t}+uu_{x}+K\gamma\rho^{\gamma-2}\rho_{x}=0\text{,}%
\end{equation}%
\begin{equation}
u_{t}+\frac{1}{2}\partial_{x}(u^{2})+\frac{K\gamma}{\gamma-1}\partial_{x}%
(\rho^{\gamma-1}-\bar{\rho}^{\gamma-1})=0\text{.}%
\end{equation}
As before, we multiply the above equation by function $f(x)$ on both sides and
take the integration with respect to $x$, yielding%
\begin{equation}
\int_{-\infty}^{+\infty}f(x)u_{t}dx+\frac{1}{2}\int_{-\infty}^{+\infty
}f(x)\partial_{x}(u^{2})+\frac{K\gamma}{\gamma-1}\int_{-\infty}^{+\infty
}f(x)\partial_{x}(\rho^{\gamma-1}-\bar{\rho}^{\gamma-1})=0\text{.}%
\end{equation}
By using the integration by parts, we obtain%
\begin{align}
&  \dot{H_{2}}(t)+\frac{1}{2}\left[  f(x)u^{2}\right]  _{-R-\sigma
t}^{R+\sigma t}+\frac{K\gamma}{\gamma-1}\left[  f(x)(\rho^{\gamma-1}-\bar
{\rho}^{\gamma-1})\right]  _{-R-\sigma t}^{R+\sigma t}\nonumber\\
&  =\frac{1}{2}\int_{-R-\sigma t}^{R+\sigma t}u^{2}{f}^{\prime}(x)dx+\frac
{K\gamma}{\gamma-1}\int_{-R-\sigma t}^{R+\sigma t}(\rho^{\gamma-1}-\bar{\rho
}^{\gamma-1}){f}^{\prime}(x)dx\text{.}%
\end{align}
Hence,%
\begin{equation}
\dot{H_{2}}(t)=\frac{1}{2}\int_{-R-\sigma t}^{R+\sigma t}u^{2}{f}^{\prime
}(x)dx+\frac{K\gamma}{\gamma-1}\int_{-R-\sigma t}^{R+\sigma t}(\rho^{\gamma
-1}-\bar{\rho}^{\gamma-1}){f}^{\prime}(x)dx \label{48}%
\end{equation}%
\begin{equation}
\geq\frac{1}{2}\int_{-R-\sigma t}^{R+\sigma t}u^{2}{f}^{\prime}(x)dx-\frac
{K\gamma}{\gamma-1}\bar{\rho}^{\gamma-1}f(R+\sigma t)\text{.}%
\end{equation}
\newline On the other hand,%
\begin{equation}
\left[  \int_{-R-\sigma t}^{R+\sigma t}uf(x)dx\right]  ^{2}\leq\left(
\int_{-R-\sigma t}^{R+\sigma t}u^{2}{f}^{\prime}(x)dx\right)  \left(
\int_{-R-\sigma t}^{R+\sigma t}\frac{f^{2}(x)}{{f}^{\prime}(x)}dx\right)
\text{.}%
\end{equation}
Then,%
\begin{equation}
\int_{-R-\sigma t}^{R+\sigma t}u^{2}{f}^{\prime}(x)dx\geq\frac{H_{2}^{2}%
(t)}{B_{2}(t)}\text{.}%
\end{equation}
Thus, we have%
\begin{equation}
\dot{H_{2}}(t)\geq\frac{H_{2}^{2}(t)}{2B_{2}(t)}-\frac{K\gamma}{\gamma-1}%
\bar{\rho}^{\gamma-1}f(R+\sigma t)\text{.}%
\end{equation}
For $0\leq t\leq\tau$, we obtain%
\begin{equation}
\dot{H_{2}}(t)\geq\frac{H_{2}^{2}(t)}{aB_{2}(t)}+\left[  \frac{(a-2)H_{2}%
^{2}(t)}{2aB_{2}(t)}-\frac{K\gamma}{\gamma-1}\bar{\rho}^{\gamma-1}f(R+\sigma
t)\right]
\end{equation}%
\begin{equation}
\geq\frac{H_{2}^{2}(t)}{aB_{2}(t)}+\left[  \frac{(a-2)H_{2}^{2}(t)}%
{2aB_{2}(\tau)}-\frac{K\gamma}{\gamma-1}\bar{\rho}^{\gamma-1}f(R+\sigma
\tau)\right]
\end{equation}%
\begin{equation}
=:\frac{H_{2}^{2}(t)}{aB_{2}(t)}+G_{2}(t)\text{.}%
\end{equation}
As before, from $G_{2}(0)>0$, we have $G_{2}(t)\geq0$ for $0\leq t\leq\tau$.
Therefore,%
\begin{equation}
\dot{H_{2}}(t)\geq\frac{H_{2}^{2}(t)}{aB_{2}(t)}\text{.}%
\end{equation}
It follows that the time $T<\tau$ if condition ($\ref{14}$) is satisfied.

The proof is complete.
\end{proof}

To give the proofs of Theorems $\ref{t7}$ and $\ref{t9}$, we need the
following lemma.

\begin{lemma}
Define $m_{1}(t)=\int_{0}^{\infty}(\rho-\bar{\rho})r^{N-1}dr$. Then we have
${m_{1}}^{\prime}(t)=0$ for $N\geq1$. In other words, $m_{1}(t)=m_{1}%
(0)$.\label{lemmaofmt}
\end{lemma}

\begin{proof}
Note that the integral is well defined by Corollary $\ref{c4}$. Thus, we have%
\begin{equation}
{m_{1}}^{\prime}(t)=\int_{0}^{\infty}\rho_{t}r^{N-1}dr
\end{equation}%
\begin{equation}
=-\int_{0}^{\infty}\left(  (V\rho)_{r}+\frac{N-1}{r}\rho V\right)  r^{N-1}dr
\end{equation}%
\begin{equation}
=-\int_{0}^{\infty}\left(  r^{N-1}(V\rho)_{r}+(\rho V)(N-1)r^{N-2}\right)  dr
\end{equation}%
\begin{equation}
=-\int_{0}^{\infty}\left(  r^{N-1}\rho V\right)  _{r}dr
\end{equation}%
\begin{equation}
=-\int_{0}^{R+\sigma t}\left(  r^{N-1}\rho V\right)  _{r}dr \label{e53}%
\end{equation}%
\begin{equation}
=-\left[  r^{N-1}\rho V\right]  _{0}^{R+\sigma t} \label{eee57}%
\end{equation}%
\begin{equation}
=0\text{,}%
\end{equation}
for $N>1$.\newline For $N=1$, expression ($\ref{eee57}$) is still zero, as by
continuity,%
\begin{equation}
V(t,0)=\lim_{x\rightarrow0^{+}}u(t,x)=\lim_{x\rightarrow0^{-}}%
u(t,x)=-V(t,0)\text{,}%
\end{equation}
which implies $V(t,0)=0$.
\end{proof}

\begin{remark}
It should be noted that function $m_{1}(t)$ in the above lemma is a radial
symmetry version of the $m(t)$ function in $\cite{SI}$.
\end{remark}

\begin{remark}
Similarly, ${m_{2}}^{\prime}(t)=0$ if $m_{2}(t)=\int_{-\infty}^{+\infty}%
(\rho(t,x)-\bar{\rho})dx$ for the $1$-dimensional Euler equations in the
non-radial symmetry case.
\end{remark}

Now, we are ready to present the proof of Theorem $\ref{t7}$.

\begin{theorem}
\label{t7}Fix $\tau>0$. Under the setting of Corollary $\ref{c4}$, we
have\newline Case 1: $\gamma\geq2$ and $m_{1}(0)\geq0$. If $H_{3}(0)$ is large
enough such that%
\begin{equation}
H_{3}(0)>\frac{2\sigma R^{N+1}(R+\sigma\tau)^{N+1}}{N[(R+\sigma\tau
)^{N+1}-R^{N+1}]}\text{,}%
\end{equation}
then the time $T<\tau$.\newline Case 2: $\gamma=2$ and $m_{1}(0)<0$. If
$H_{3}(0)$ is large enough such that%
\begin{equation}
H_{3}(0)>\frac{a\sigma R^{N+1}(R+\sigma\tau)^{N+1}}{N[(R+\sigma\tau
)^{N+1}-R^{N+1}]}\text{,} \label{58}%
\end{equation}
then the time $T<\tau$,\newline where%
\begin{equation}
H_{3}(t)=\int_{0}^{\infty}r^{N}V(t,r)dr\text{,}%
\end{equation}%
\begin{equation}
m_{1}(t)=\int_{0}^{\infty}(\rho(t,r)-\bar{\rho})r^{N-1}dr
\end{equation}
and%
\begin{equation}
a=1+\sqrt{1+\frac{-4N^{2}Km_{1}(0)[(R+\sigma\tau)^{N+1}-R^{N+1}]^{2}%
}{(N+1)\sigma^{2}R^{2N+2}(R+\sigma\tau)^{N}}}\text{.} \label{61}%
\end{equation}

\end{theorem}

\begin{proof}
For function $f(r)=r^{N}$, equation ($\ref{28}$) becomes%
\begin{equation}
\dot{H_{3}}(t)=\frac{N}{2}\int_{0}^{R+\sigma t}V^{2}r^{N-1}dr+\frac{KN\gamma
}{\gamma-1}\int_{0}^{R+\sigma t}(\rho^{\gamma-1}-\bar{\rho}^{\gamma-1}%
)r^{N-1}dr\text{.}%
\end{equation}
The Cauchy Inequality can be applied to confirm that%
\begin{equation}
H_{3}^{2}(t)\leq\frac{(R+\sigma t)^{N+2}}{(N+1)}\int_{0}^{R+\sigma t}%
V^{2}r^{N-1}dr\text{.}%
\end{equation}
Thus,%
\begin{equation}
\dot{H_{3}}(t)\geq\frac{N(N+1)}{2(R+\sigma t)^{N+2}}H_{3}^{2}(t)+\frac
{KN\gamma}{\gamma-1}\int_{0}^{R+\sigma t}(\rho^{\gamma-1}-\bar{\rho}%
^{\gamma-1})r^{N-1}dr\text{.} \label{e64}%
\end{equation}
\newline For $\gamma>2$, it can be shown by Holder's Inequality that the
second term on the right-hand side of the above equation is greater than or
equal to zero. More precisely, for $\gamma>2$, as $m_{1}(t)=m_{1}(0)\geq0$, we
have%
\begin{equation}
\int_{0}^{R+\sigma t}\bar{\rho}r^{N-1}dr\leq\int_{0}^{R+\sigma t}\rho
r^{N-1}dr\leq\left(  \int_{0}^{R+\sigma t}\rho^{\gamma-1}r^{N-1}dr\right)
^{\frac{1}{\gamma-1}}\left(  \int_{0}^{R+\sigma t}(1)r^{N-1}dr\right)
^{1-\frac{1}{\gamma-1}}\text{.}%
\end{equation}
It follows that%
\begin{equation}
\bar{\rho}^{\gamma-1}\int_{0}^{R+\sigma t}r^{N-1}dr\leq\int_{0}^{R+\sigma
t}\rho^{\gamma-1}r^{N-1}dr
\end{equation}%
\begin{equation}
\int_{0}^{R+\sigma t}(\rho^{\gamma-1}-\bar{\rho}^{\gamma-1})r^{N-1}%
dr\geq0\text{.}%
\end{equation}
\newline For $\gamma=2$, equation ($\ref{e64}$) becomes%
\begin{equation}
\dot{H_{3}}(t)\geq\frac{N(N+1)}{2(R+\sigma t)^{N+2}}H_{3}^{2}(t)+2KNm_{1}%
(0)\text{.} \label{e68}%
\end{equation}
For $\gamma\geq2$ and $m_{1}(0)\geq0$, we have%
\begin{equation}
\dot{H_{3}}(t)\geq\frac{N(N+1)}{2(R+\sigma t)^{N+2}}H_{3}^{2}(t)\text{.}%
\end{equation}
As $H_{3}(0)>0$, we have $H_{3}(t)\geq0$ for $t\geq0$ and%
\begin{equation}
\frac{1}{H_{3}(0)}-\frac{1}{H_{3}(t)}\leq\int_{0}^{t}\frac{N(N+1)}{2(R+\sigma
s)^{N+2}}ds\text{.}%
\end{equation}
Therefore, for $0\leq t\leq\tau$, we have%
\begin{equation}
\frac{1}{H_{3}(0)}-\frac{1}{H_{3}(t)}\leq\int_{0}^{\tau}\frac{N(N+1)}%
{2(R+\sigma s)^{N+2}}ds=\frac{N\left[  (R+\sigma\tau)^{N+1}-R^{N+1}\right]
}{2\sigma R^{N+1}(R+\sigma\tau)^{N+1}}\text{.}%
\end{equation}
The result of Case 1 follows.

For Case 2, from equation ($\ref{e68}$), we have%
\begin{equation}
\dot{H_{3}}(t)\geq\frac{N(N+1)}{a(R+\sigma t)^{N+2}}H_{3}^{2}(t)+\left[
\frac{(a-2)N(N+1)}{2a(R+\sigma t)^{N+2}}H_{3}^{2}(t)+2KNm_{1}(0)\right]
\end{equation}%
\begin{equation}
\geq\frac{N(N+1)}{a(R+\sigma t)^{N+2}}H_{3}^{2}(t)+\left[  \frac
{(a-2)N(N+1)}{2a(R+\sigma\tau)^{N+2}}H_{3}^{2}(t)+2KNm_{1}(0)\right]
\end{equation}%
\begin{equation}
=:\frac{N(N+1)}{a(R+\sigma t)^{N+2}}H_{3}^{2}(t)+G_{3}(t)\text{,}%
\end{equation}
for $0\leq t\leq\tau$.\newline Suppose%
\begin{equation}
G_{3}(0)>0\text{,}%
\end{equation}
or equivalently,%
\begin{equation}
H_{3}^{2}(0)>\frac{-4aKm_{1}(0)(R+\sigma\tau)^{N+2}}{(a-2)(N+1)}\text{,}
\label{e76}%
\end{equation}
where the value of $a>2$ will be determined later.\newline As $G_{3}(0)>0$, we
have $G_{3}(t)\geq0$ and%
\begin{equation}
\dot{H_{3}}(t)\geq\frac{N(N+1)}{a(R+\sigma t)^{N+2}}H_{3}^{2}(t)\text{,}%
\end{equation}
for $0\leq t\leq\tau$.\newline Hence,%
\begin{equation}
0<\frac{1}{H_{3}(t)}\leq\frac{1}{H_{3}(0)}-\int_{0}^{t}\frac{N(N+1)}%
{a(R+\sigma s)^{N+2}}ds
\end{equation}%
\begin{equation}
=\frac{1}{H_{3}(0)}-\frac{N[(R+\sigma t)^{N+1}-R^{N+1}]}{a\sigma
R^{N+1}(R+\sigma t)^{N+1}}%
\end{equation}
for $0\leq t\leq\tau$.\newline Then, we have the time $T<\tau$ if%
\begin{equation}
H_{3}(0)\geq\frac{a\sigma R^{N+1}(R+\sigma\tau)^{N+1}}{N[(R+\sigma\tau
)^{N+1}-R^{N+1}]}\text{.} \label{Eq78}%
\end{equation}
By solving the following equation for $a>2$ for the equation%
\begin{equation}
\frac{a\sigma R^{N+1}(R+\sigma\tau)^{N+1}}{N[(R+\sigma\tau)^{N+1}-R^{N+1}%
]}=\sqrt{\frac{-4aKm_{1}(0)(R+\sigma\tau)^{N+2}}{(a-2)(N+1)}}\text{,}
\label{811}%
\end{equation}
we obtain the value of $a$ in equation ($\ref{61}$) and hence condition
($\ref{58}$) implies conditions ($\ref{e76}$) and ($\ref{Eq78}$). The proof of
Case 2 is complete.
\end{proof}

Next, we have the following theorem for the $1$-dimensional Euler equations
$(\ref{Euler})$ in non-radial symmetry.

\begin{theorem}
\label{t9}Under the setting of Lemma $\ref{tt2}$ with $N=1$ and $\gamma\geq2$,
if%
\begin{equation}
H_{4}(0)>\frac{8\sigma R^{2}}{3} \label{21}%
\end{equation}
and%
\begin{equation}
m_{2}(0)\geq0\text{,}%
\end{equation}
then the $C^{1}$ non-vacuum solutions blow up on a finite time $T_{1}$, where%
\begin{equation}
H_{4}(t)=\int_{-\infty}^{+\infty}xu(t,x)dx \label{eq10}%
\end{equation}
and%
\begin{equation}
m_{2}(t)=\int_{-\infty}^{+\infty}(\rho(t,x)-\bar{\rho})dx\text{.}%
\end{equation}

\end{theorem}

\begin{proof}
For function $f(x)=x$, equation ($\ref{48}$) becomes%
\begin{equation}
\dot{H_{4}}(t)=\frac{1}{2}\int_{-R-\sigma t}^{R+\sigma t}u^{2}dx+\frac
{K\gamma}{\gamma-1}\int_{-R-\sigma t}^{R+\sigma t}(\rho^{\gamma-1}-\bar{\rho
}^{\gamma-1})dx\text{.} \label{72}%
\end{equation}
For $\gamma=2$, the second term on the right-hand side of the above equation
is $\frac{K\gamma}{\gamma-1}m_{2}(0)$, which is greater than or equal to
zero.\newline For $\gamma>2$, it can be shown by Holder's inequality that the
second term on the right-hand side of the above equation is greater than or
equal to zero. More precisely, for $\gamma>2$, as $m_{2}(t)=m_{2}(0)\geq0$, we
have%
\begin{equation}
\int_{-R-\sigma t}^{R+\sigma t}\bar{\rho}dx\leq\int_{-R-\sigma t}^{R+\sigma
t}\rho dx\leq\left(  \int_{-R-\sigma t}^{R+\sigma t}\rho^{\gamma-1}dx\right)
^{\frac{1}{\gamma-1}}\left(  \int_{-R-\sigma t}^{R+\sigma t}(1)dx\right)
^{1-\frac{1}{\gamma-1}}\text{.}%
\end{equation}
It follows that%
\begin{equation}
\bar{\rho}^{\gamma-1}\int_{-R-\sigma t}^{R+\sigma t}dx\leq\int_{-R-\sigma
t}^{R+\sigma t}\rho^{\gamma-1}dx
\end{equation}%
\begin{equation}
\int_{-R-\sigma t}^{R+\sigma t}(\rho^{\gamma-1}-\bar{\rho}^{\gamma-1}%
)dx\geq0\text{.}%
\end{equation}
Thus,%
\begin{equation}
\dot{H_{4}}(t)\geq\frac{1}{2}\int_{-R-\sigma t}^{R+\sigma t}u^{2}dx\text{.}%
\end{equation}
\newline The Cauchy Inequality can be used to check%
\begin{equation}
H_{4}^{2}(t)\leq\left(  \int_{-R-\sigma t}^{R+\sigma t}u^{2}dx\right)  \left(
\frac{2(R+\sigma t)^{3}}{3}\right)  \text{.} \label{e87}%
\end{equation}
Thus,%
\begin{equation}
\dot{H_{4}}(t)\geq\frac{3H_{4}^{2}(t)}{4(R+\sigma t)^{3}}\text{.} \label{75}%
\end{equation}
Hence,%
\begin{equation}
0<\frac{1}{H_{4}(t)}\leq\frac{1}{H_{4}(0)}-\frac{3}{8\sigma}\left[  \frac
{1}{R^{2}}-\frac{1}{(R+\sigma t)^{2}}\right]  \text{.} \label{76}%
\end{equation}
If the solutions are global, then by letting $t\rightarrow\infty$, we have%
\begin{equation}
0\leq\frac{1}{H_{4}(0)}-\frac{3}{8\sigma R^{2}}\text{,}%
\end{equation}
which contradicts condition ($\ref{21}$).\newline The proof is complete.
\end{proof}

Last, we present the following corollary, which is easily obtained from the
proof of Theorem $\ref{t9}$.

\begin{corollary}
Fix $\tau>0$. Under the setting of Lemma $\ref{tt2}$ with $N=1$ and
$\gamma\geq2$, we have\newline Case $1$: $\gamma\geq2$ and $m_{2}(0)\geq0$. If
$H_{4}(0)$ is large enough such that%
\begin{equation}
H_{4}(0)\geq\frac{8R^{2}(R+\sigma\tau)^{2}}{3\tau(2R+\sigma\tau)}\text{,}%
\end{equation}
then the time $T<\tau$.\newline Case $2$: $\gamma=2$ and $m_{2}(0)<0$. If
$H_{4}(0)$ is large enough such that%
\begin{equation}
H_{4}(0)>\frac{2aR^{2}(R+\sigma\tau)^{2}}{\tau(2R+\sigma\tau)}\text{,}
\label{95}%
\end{equation}
then the time $T<\tau$, where%
\begin{equation}
a=\frac{2}{3}+\sqrt{\frac{4}{9}-\frac{6Km_{2}(0)\tau^{2}(2R+\sigma\tau)^{2}%
}{9R^{4}(R+\sigma\tau)}}\text{.} \label{96}%
\end{equation}

\end{corollary}

\begin{proof}
The result of Case 1 follows from equation ($\ref{76}$).

For Case 2, from equation ($\ref{72}$), we have%
\begin{equation}
\dot{H_{4}}(t)=\frac{1}{2}\int_{-R-\sigma t}^{R+\sigma t}u^{2}dx+2Km_{2}%
(0)\text{.}%
\end{equation}
From equation ($\ref{e87}$), we have%
\begin{equation}
\int_{-R-\sigma t}^{R+\sigma t}u^{2}dx\geq\frac{3H_{4}^{2}(t)}{2(R+\sigma
t)^{3}}\text{.}%
\end{equation}
Thus,%
\begin{equation}
\dot{H_{4}}(t)\geq\frac{3H_{4}^{2}(t)}{4(R+\sigma t)^{3}}+2Km_{2}(0)
\end{equation}%
\begin{equation}
=\frac{H_{4}^{2}(t)}{a(R+\sigma t)^{3}}+\left[  \frac{3a-4}{4a}\frac{H_{4}%
^{2}(t)}{(R+\sigma t)^{3}}+2Km_{2}(0)\right]
\end{equation}%
\begin{equation}
\geq\frac{H_{4}^{2}(t)}{a(R+\sigma t)^{3}}+\left[  \frac{3a-4}{4a}\frac
{H_{4}^{2}(t)}{(R+\sigma\tau)^{3}}+2Km_{2}(0)\right]
\end{equation}%
\begin{equation}
:=\frac{H_{4}^{2}(t)}{a(R+\sigma t)^{3}}+G_{4}(t)
\end{equation}
for $0\leq t\leq\tau$.\newline Suppose
\begin{equation}
G_{4}(0)>0
\end{equation}
or equivalently,%
\begin{equation}
H_{4}^{2}(0)>\frac{-8aKm_{2}(0)(R+\sigma\tau)^{3}}{(3a-4)}\text{,} \label{104}%
\end{equation}
where the value of $a>4/3$ will be determined later.\newline As before, we
have $G_{4}(t)\geq0$ and%
\begin{equation}
\dot{H_{4}}(t)\geq\frac{H_{4}^{2}(t)}{a(R+\sigma t)^{3}}%
\end{equation}
for $0\leq t\leq\tau$.\newline Therefore, the time $T<\tau$ if
\begin{equation}
H_{4}(0)\geq\frac{2aR^{2}(R+\sigma\tau)^{2}}{\tau(2R+\sigma\tau)}\text{.}
\label{106}%
\end{equation}
Now, solving the equation in $a>4/3$ for the equation%
\begin{equation}
\frac{2aR^{2}(R+\sigma\tau)^{2}}{\tau(2R+\sigma\tau)}=\sqrt{\frac
{-8aKm_{2}(0)(R+\sigma\tau)^{3}}{(3a-4)}}\text{,}%
\end{equation}
we obtain the value of $a$ in equation ($\ref{96}$). Hence condition
($\ref{95}$) implies conditions ($\ref{104}$) and ($\ref{106}$).\newline The
proof is complete.
\end{proof}

\section{Conclusions}

In this article, we provide several new blowup results for the Euler equations
($\ref{Euler}$) for $N=1$ and general $N$-dimensional Euler equations in
radial symmetry ($\ref{eqRS}$) with initial non-vacuum conditions.
Specifically, we show that if the initial function $H_{i}(0)$ is large enough,
then blowup occurs on or before a finite time and the corresponding blowup
time can be estimated. In particular, the new class of testing functions in
Theorem $\ref{tt6}$ consists of general, non-negative, strictly increasing
$C^{1\text{ }}$functions $f(r)$. This is our main contribution.

The similar analysis can be applied to obtain the corresponding blowup results
for the compressible Euler equations with linear damping.

\section{Acknowledgement}

The research in this paper was partially supported by the Dean's Research
Grant FLASS/ECR-9 from the Hong Kong Institute of Education. We thank for the
reviewers' comments for improving this article.

\section*{Appendix}

\textbf{Proof of Lemma \ref{tt2}.}\newline Define%
\begin{equation}
v=\frac{2}{\gamma-1}\left(  \sqrt{{P}^{\prime}(\rho)}-\sigma\right)  \text{,}%
\end{equation}
where $P$ is regarded as a function of $\rho$.\newline Then, equation
($\ref{Euler}$)$_{1}$ is transformed into%
\begin{equation}
v_{t}+\sigma{\nabla}\cdot u=-u\cdot{\nabla}v-\frac{\gamma-1}{2}v{\nabla}\cdot
u \label{ee114}%
\end{equation}
and equation ($\ref{Euler}$)$_{2}$ is transformed into%
\begin{equation}
u_{t}+\sigma{\nabla}v=-(u\cdot{\nabla})u-\frac{\gamma-1}{2}v{\nabla}v\text{.}
\label{ee115}%
\end{equation}
Multiply equation ($\ref{ee114}$) by $v$ and equation ($\ref{ee115}$) by $u$.
Then, add them together and rearrange the terms to get%
\begin{equation}
\left(  \frac{v^{2}+|u|^{2}}{2}\right)  _{t}+{\nabla}\cdot(\sigma
vu)=-vu\cdot{\nabla}v-u\cdot(u\cdot{\nabla}u)-\frac{\gamma-1}{2}v^{2}{\nabla
}\cdot u-\frac{\gamma-1}{2}vu\cdot{\nabla}v\text{,} \label{ee116}%
\end{equation}
where $u\cdot{\nabla}u:=\sum_{i=1}^{N}u_{i}{\nabla}u_{i}$ and $u=(u_{1}%
,u_{2},\cdots,u_{N})$.\newline Fix $(x,t)\in\mathbb{R}^{N}\times(0,T]$ and
$\mu\in\lbrack0,t)$. Define the truncated cone%
\begin{equation}
C_{\mu}:=\{(y,s):|y-x|\leq\sigma(t-s),0\leq s\leq\mu\}\text{.}%
\end{equation}
Note that the cross sections of $C_{\mu}$ are%
\begin{equation}
U(s):=\{y:|y-x|\leq\sigma(t-s)\}\text{ \ \ \ for }s\in\lbrack0,\mu].\newline%
\end{equation}
Lastly, define%
\begin{equation}
e(s):=\int_{U(s)}\frac{v^{2}+|u|^{2}}{2}(s,y)dy\text{.} \label{ee119}%
\end{equation}
Take the integration on both sides of equation ($\ref{ee116}$) over $C_{\mu}$
to get%
\begin{align}
&  \int_{0}^{\mu}\int_{U(s)}\left[  \left(  \frac{v^{2}+|u|^{2}}{2}\right)
_{t}+{\nabla}\cdot(\sigma vu)\right]  dyds\label{ee120}\\
&  =\int_{0}^{\mu}\int_{U(s)}\left[  -vu\cdot{\nabla}v-u\cdot(u\cdot{\nabla
}u)-\frac{\gamma-1}{2}v^{2}{\nabla}\cdot u-\frac{\gamma-1}{2}vu\cdot{\nabla
}v\right]  dyds\text{.} \label{ee121}%
\end{align}
\textbf{Step 1.} Applying the Differentiation Formula for Moving Regions, the
Fundamental Theorem of Calculus and the Divergence Theorem, expression
($\ref{ee120}$) is equal to%
\begin{align}
&  \int_{U(\mu)}\left(  \frac{v^{2}+|u|^{2}}{2}\right)  (\mu,y)dy-\int
_{U(0)}\left(  \frac{v^{2}+|u|^{2}}{2}\right)  (0,y)dy+\int_{0}^{\mu}%
\int_{\partial U(s)}\left[  \sigma\left(  \frac{v^{2}+|u|^{2}}{2}\right)
+\frac{y-x}{|y-x|}\cdot\sigma vu\right]  dSds\\
&  =e(\mu)-e(0)+\sigma\int_{0}^{\mu}\int_{\partial U(s)}\left(  \frac
{v^{2}+|u|^{2}}{2}+\frac{y-x}{|y-x|}\cdot vu\right)  dSds\\
&  \geq e(\mu)-e(0)\text{,}%
\end{align}
where $dS$ is the surface element with respect to the variable $y$ and
$\partial U(s)$ is the boundary of $U(s)$. Note that by the Cauchy Inequality,%
\begin{equation}
\frac{y-x}{|y-x|}\cdot vu\leq\left\vert \frac{y-x}{|y-x|}\cdot vu\right\vert
\leq|vu|=|v||u|\leq\frac{v^{2}+|u|^{2}}{2}\text{.}%
\end{equation}
\textbf{Step 2.} By the Cauchy Inequality and the following two inequalities,%
\begin{equation}
|u\cdot{\nabla}u|\leq|u|\sqrt{\sum_{i=1}^{N}|{\nabla}u_{i}|^{2}}\qquad\text{
and }\qquad|{\nabla}\cdot u|\leq\sqrt{\sum_{i=1}^{N}|{\nabla}u_{i}|^{2}},
\end{equation}
the integrand of $(\ref{ee121})$ can be estimated as follows:%
\begin{align}
&  -vu\cdot{\nabla}v-u\cdot(u\cdot{\nabla}u)-\frac{\gamma-1}{2}v^{2}{\nabla
}\cdot u-\frac{\gamma-1}{2}vu\cdot{\nabla}v\\
&  \leq|v||u||{\nabla}v|+|u||(u\cdot{\nabla}u)|+\frac{\gamma-1}{2}%
v^{2}|{\nabla}\cdot u|+\frac{\gamma-1}{2}|v||u||{\nabla}v|\\
&  \leq\gamma\left[  |v||u||{\nabla}v|+|u||(u\cdot{\nabla}u)|+v^{2}|{\nabla
}\cdot u|\right] \\
&  \leq\gamma\left[  \frac{v^{2}+|u|^{2}}{2}|{\nabla}v|+\frac{v^{2}+|u|^{2}%
}{2}\left(  2\sqrt{\sum_{i=1}^{N}|{\nabla}u_{i}|^{2}}\right)  \right] \\
&  =\gamma\left(  \frac{v^{2}+|u|^{2}}{2}\right)  \left(  |{\nabla v}%
|+2\sqrt{\sum_{i=1}^{N}|{\nabla}u_{i}|^{2}}\right)  .
\end{align}
Thus, expression ($\ref{ee121}$) is less than or equal to%
\[
C\int_{0}^{\mu}e(s)ds\text{,}%
\]
where%
\begin{equation}
C=\gamma\max_{C_{\mu}}\left\{  |{\nabla}v|+2\sqrt{\sum_{i=1}^{N}|{\nabla}%
u_{i}|^{2}}\right\}  <+\infty.
\end{equation}
\textbf{Step 3.} Combining the results of Step 1 and Step 2 produces%
\begin{equation}
e(\mu)-e(0)\leq C\int_{0}^{\mu}e(s)ds\text{.}%
\end{equation}
By Gronwall's Inequality and the definition ($\ref{ee119}$) of $e(s)$, we see
that%
\begin{equation}
0\leq e(\mu)\leq e(0)\exp(Ct)\text{.}%
\end{equation}
If $|x|>R+\sigma t$, $|y|>R$ for $y\in U(0)$.\newline Thus, $e(0)=0$ and
$e(\mu)=0$ for $|x|>R+\sigma t$.\newline Thus, $v(\mu,x)=u(\mu,x)=0$ for
$|x|>R+\sigma t$.\newline Thus, $(\rho,u)(\mu,x)=(\bar{\rho},0)$ for
$|x|>R+\sigma t$.\newline As $\mu\in\lbrack0,t)$ is arbitrary, the result
follows by continuity.

\end{document}